\documentclass{amsart}
\usepackage{amsmath}
\usepackage{amsthm}
\usepackage{amsfonts}
\usepackage{amssymb}
\usepackage{caption}
\usepackage{subcaption}
\usepackage{graphicx}
\usepackage{setspace}
\usepackage{tikz}
\usepackage{tikz-cd}
\usepackage{thmtools, thm-restate}
\usepackage{float}

\graphicspath{{images/}}

\newtheorem{thm}{Theorem}[section]

\theoremstyle{definition}
\newtheorem{definition}[thm]{Definition}

\theoremstyle{remark}

\numberwithin{equation}{section}

\begin{document}

\title{A note on pure braids and link concordance}

\author{Miriam Kuzbary}
\address{School of Mathematics, Georgia Institute of Technology, Atlanta, GA 30332}
\email{kuzbary@gatech.edu}
\urladdr{http://people.math.gatech.edu/~mkuzbary3/}

\begin{abstract}

The knot concordance group can be contextualized as organizing problems about 3- and 4-dimensional spaces and the relationships between them. Every 3-manifold is surgery on some link, not necessarily a knot, and thus it is natural to ask about such a group for links. In 1988, Le Dimet constructed the string link concordance groups and in 1998, Habegger and Lin precisely characterized these groups as quotients of the link concordance sets using a group action. Notably, the knot concordance group is abelian while, for each $n$, the string link concordance group on $n$ strands is non-abelian as it contains the pure braid group on $n$ strands as a subgroup. In this work, we prove that even the quotient of each string link concordance group by its pure braid subgroup is non-abelian.

\end{abstract}

 \maketitle


\section{Introduction}

A knot is a smooth embedding of $S^1$ into $S^3$ and an $n$-component link is a smooth embedding of $n$ disjoint copies of $S^1$ into $S^3$. The set of knots modulo a $4$-dimensional equivalence relation called concordance forms the celebrated knot concordance group $\mathcal{C}$ using the operation connected sum, while the set of $n$-component links modulo concordance does not have a well-defined connected sum (even when the links are ordered and oriented) \cite{foxmilnor}. As a result, multiple notions of link concordance group have arisen in the literature with somewhat different structural properties  \cite{hosokawa} \cite{ledimet} \cite{donaldowens}. 

We focus on the $n$-strand string link concordance group $\mathcal{C}(n)$ of Le Dimet as it is the only notion of link concordance group which is non-abelian. This is in stark contrast to the original knot concordance group which is abelian, and thus is perhaps a further indicator of the complexity of the link concordance story. While each link has an infinite number of distinct string link representatives, the correspondence between links and their string link representatives was characterized by Habegger and Lin using the action of $\mathcal{C}(2n)$ on $\mathcal{C}(n)$ \cite{habeggerlinconc}. In particular, a string link is trivial in this group exactly when the link in $S^3$ it represents is concordant to the unlink. 

The string link concordance group in the $n > 2$ case is known to be non-abelian as it contains $\mathcal{P}(n)$, the pure braid group on $n$ strands, as a proper subgroup \cite{ledimet} (for $n=2$ see \cite{decampos}). Therefore, it is not clear whether $\mathcal{C}(n)$ being non-abelian is merely an inherited property from its pure braid subgroup. This question can be viewed as in some sense complementary to the question of classifying links up to link homotopy answered in \cite{habeggerlinhtpy}, as every string link is link homotopic to a pure braid. Understanding how $\mathcal{P}(n)$ sits inside $\mathcal{C}(n)$ is made more difficult by the fact that $\mathcal{P}(n)$ is not a normal subgroup for $n>2$ \cite{klw}. Despite this difficulty, we have proven the following theorem.

\begin{restatable}{thm}{nonabelian}\label{nonabelian}
$\frac{\mathcal{C}(n)}{Ncl(\mathcal{P}(n))}$ is non-abelian for all $n$.
\end{restatable}

It is natural to further ask about the structure of the quotient group $\frac{\mathcal{C}(n)}{Ncl(\mathcal{P}(n))}$ and we have made the following observation. 

\begin{restatable}{prop}{knotconcsubgroup}
$\mathcal{C}(n)/Ncl(\mathcal{P}(n))$ has a central subgroup isomorphic to $ \mathop{\oplus}\limits^n \ \mathcal{C}$. 
\end{restatable}

Notice that Theorem \ref{nonabelian} leads to questions about how close this quotient is to being abelian. More precisely, it is natural to ask if this quotient solvable or even nilpotent?

\begin{restatable}{conj}{solvable}
$\mathcal{C}(n)$ is not solvable for any $n$ and neither is $\frac{\mathcal{C}(n)}{Ncl(\mathcal{P}(n))}$. 
\end{restatable}


%


There are many natural questions arising from these results, including:

\begin{itemize}
\item Does $\frac{\mathcal{C}(n)}{Ncl(\mathcal{P}(n))}$ contain elements of finite order greater than two?
\item What is the abelianization of $\frac{\mathcal{C}(n)}{Ncl(\mathcal{P}(n))}$?
\item How do the images of the $n$-solvable filtration and $n$-bipolar filtration behave in this quotient?
\item What can be said about the structure of $Ncl(\mathcal{P}(n)$ itself and which properties of $\mathcal{P}(n)$ are preserved in this closure?
\end{itemize}

\section{Proofs and tools used in proofs}

In this section, we carefully define $n$-component string links and the equivalence relation on them required to form a group. We then outline the main invariants used obstruct the relevant quotient being abelian and present proofs of the results.

\begin{definition}[$n$-component string link] \label{def:stringlink}
Let $D$ be the unit disk, $I$ the unit interval, and $\{p_1, p_2, ..., p_k\}$ be $n$ points in the interior of $D$. An \textit{$n$-component (pure) string link} is a smooth proper embedding $\sigma : \bigsqcup_m I \rightarrow D \times I$ such that
\begin{align*}
\sigma|_{I_i(0)} =& \{p_i\}\times \{0\} \\
\sigma|_{I_i(1)} =& \{p_i\}\times \{1\} 
\end{align*}
\end{definition}

As is typical in the study of knots and links, we will often abuse notation and use $\sigma$ to refer to both the map itself and to its image embedded in $D^2 \times I$. Notice that a $n$-component string link is an $n$-strand pure braid when each slice $\sigma_{D^2 \times \{t\}}$ is the $n$-punctured disk. Just as we can take the closure of a braid, we can take the closure of a string link  and get a link in $S^3$. Note that not every link is the closure of a pure braid, however, every link is the closure of a pure string link in the following way.  First, introduce an embedded disk intersecting each component of $L$ once positively. Then, cut open $S^3$ along this disk to get a collection of $n$ knotted arcs in $D^2 \times I$. In order to multiply string links $\sigma_1$ and $\sigma_2$, glue $D^2 \times \{1\}  \subset \sigma_1$ to $D^2 \times \{0\}  \subset \sigma_2$ by the identity map on $D^2$. 

Geometrically, this multiplication is exactly what one would obtain from taking two closed $n$-component links $L_1$ and $L_2$ in $S^3$, choosing basing disks $\Delta_1$ and $\Delta_2$ for each link (which would give $\sigma_1$ and $\sigma_2$ if one were to cut open along them), and then band summing the $i^{th}$ strand of $L_1$ with the $i^{th}$ strand of $L_2$ exactly at the points $(L_1)_i \cap \Delta_1$ and $(L_1)_i \cap \Delta_2$. Note that the set of $n$ component pure string links merely forms a monoid and not a group, and therefore to define inverses we consider the following notion. 


\begin{definition}[String link concordance]
Two $n$-component string links $\sigma_1$ and $\sigma_2$ are \textit{concordant} if there is a smooth embedding $H: \bigsqcup_m (I \times I) \rightarrow B^3 \times I$ which is transverse to the boundary such that
\begin{align*}
H|_{(\bigsqcup_m I \times \{0\})} =& \sigma_1 \\
H|_{(\bigsqcup_m I \times \{1\})} =& \sigma_2 \\
H|_{(\bigsqcup_m \partial I \times I)} =& j_0 \times id_I 
\end{align*}
 with $j_0: \bigsqcup_m \partial I \rightarrow S^2$.
\end{definition}

Now, we can see the inverse of a string link $\sigma$ is simply $\sigma$ reflected across $D^2 \times \{1\}$. More precisely, if $r: I \times I$ is the reflection map $r(t)=1-t$, $R_t = \text{id}_{D^2} \times r$, and $R_s = \bigsqcup_{i=1}^{n} r$, then $\sigma^{-1} = R_s \circ \sigma \circ R_t$. It is clear through taking string link closures that concordant string links close to form concordant links in $S^3$.


\subsection{Milnor's Invariants}

The main result in this work was proven using Milnor's invariants. These integer-valued concordance invariants first defined by Milnor in 1957 can be interpreted as higher order linking numbers and computed in many different ways. Roughly, these invariants detect how deep the longitudes of an $n$-component link $L \in G=\pi_1(S^3\setminus \nu(L),*)$ lie in the lower central series of $G_q$ the link group of the link group $G$. The mechanism for doing so involves comparing the nilpotent quotients $G/G_q$ to those of the free group $F$ on $n$ letters and using combinatorial group theoretic techniques to analyze the longitudes in these quotients. In a sense, we can view this process as taking the ``Taylor expansion" in non-commuting variables of each longitude using Fox Calculus; the Milnor's invariants are exactly the coefficients of the resulting polynomial modulo a subset of lower order coefficients. It is certainly not obvious that such things would be concordance invariants. It is a deep result that the entire nilpotent quotients $G/G_q$ themselves are link concordance invariants and that Milnor's invariants extract concordance data from these quotients. 

These invariants are difficult to compute; however, due to Cochran we have an algorithm for computing the non-vanishing Milnor's invariants of lowest weight (which should be thought of as the lowest degree coefficients of the aforementioned polynomials) using surface systems as well as many useful theorems involving these invariants of lowest weight. The main idea of Cochran's construction relies on the celebrated theorem of Turaev and Porter which relates Milnor's invariants to Massey products on the link exterior. Cochran's process dualizes the Massey product defining system machinery to use intersections of surfaces instead of cup products of cochains. For details of the procedure, see \cite{cochranmemoir}.

%

\subsection{Proofs}

\nonabelian*
\begin{proof}
This proof relies on using the lowest weight non-vanishing Milnor invariants of a pure braid. We will first show by contradiction that if a pure braid has all pairwise linking numbers 0, the lowest weight non-vanishing Milnor's invariant of the closure of any pure braid must have at least three distinct indices. 

Let $P$ be an $n$-stranded pure braid with pairwise linking 0 whose lowest-weight non-vanishing Milnor's invariant $\mu_{\widehat{P}}(I)$ only involves two indices, i and j. Now, notice that since the Milnor's invariants of $\widehat{P}$ involving exactly two indices only depend on the 2-component sub-braids of $P$, the lowest weight non-vanishing Milnor's invariant of the sublink $\widehat{P^{\prime}}$ corresponding to the closure of the $i^{th}$ and $j^{th}$ strand of $P$ must be $\mu_{\widehat{P^{\prime}}}(I)=\mu_{\widehat{P}}(I)$. This is our contradiction, as $P^{\prime}$ is a two-stranded pure braid with vanishing linking number and therefore must be the trivial braid as linking number gives an isomorphism from $\mathcal{P}(2)$ to the integers.

%
%
%
%
%
%
%

Now, let $b \in Ncl(\mathcal{P}(n))$ with all pairwise linking numbers zero whose lowest-weight non-vanishing Milnor's invariant $\mu_b(I)$ has multi-index $I$ only involving two indices, i and j. We see that $b = \Pi_{l=1}^m s_lp_ls_l^{-1}$ where $p_l \in \mathcal{P}(n)$ and $s_l \in \mathcal{C}(n)$ for all $l$ by definition. Therefore, by Theorem 8.12 in \cite{cochranmemoir} we see that  $\mu_{\Pi_{l=1}^m s_lp_ls_l^{-1}}(I)=\Sigma_{l=1}^m \mu_{s_lp_ls_l^{-1}}(I)$.

Again, $\mu_{s_lp_ls_l^{-1}}(I)$ will depend only on the $2$-component sublinks of each string link $s_lp_ls_l^{-1}$ made up of the $i^{th}$ and $j^{th}$ components of $s_lp_ls_l^{-1}$. Each of these sublinks will be a product of the $i^{th}$ and $j^{th}$ components of $s_l$, $p_l$, and $s_l^{-1}$. We will show what happens for a specific $l$. Call these components $L$, $P^{\prime}$, and $-L$.

\begin{figure}[H]
\centering
\begin{subfigure}{.45\textwidth}
\centering
\includegraphics[scale=.15]{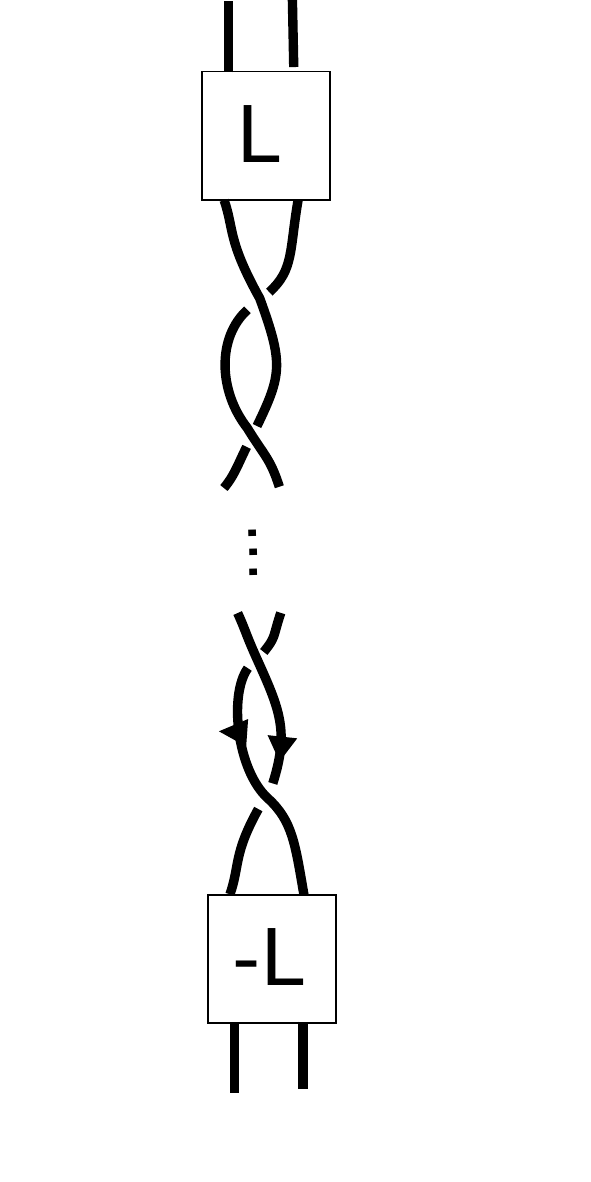}
\end{subfigure}
\begin{subfigure}{.45\textwidth}
\centering
\includegraphics[scale=.15]{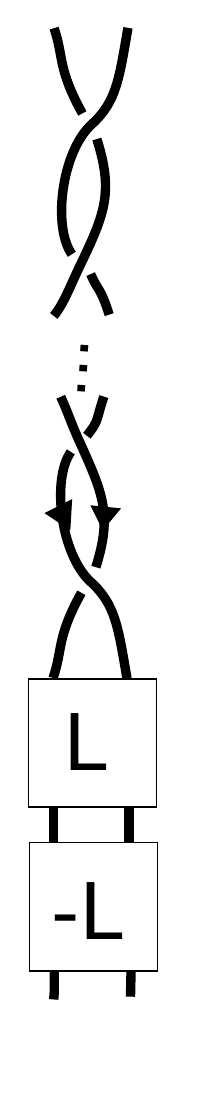}
\end{subfigure}
\caption{Pushing the string link $L$ along the full twists $P^{\prime}$.}
\label{fig:flype}
\end{figure}

As is clear in Figure \ref{fig:flype}, $P^{\prime}$ is a product of full twists and thus we can slide $L$ along the twisting region to obtain the word $PLL^{-1}$. We see this string link is concordant to $P$ and thus $\mu_{s_lp_ls_l^{-1}}(I)=\mu_{P}(I)=0$. Finally, this tells us $\mu_b(I)=0$.

Hence to prove the theorem it will suffice to show a commutator in $\mathcal{C}(2)$ has first non-vanishing Milnor invariant corresponding to a multi-index only involving two indices. Note that any commutator of string links has pairwise linking number $0$. Such a commutator appears in Proposition 4.9 of \cite{meilhanyasuhara}. This commutator of string links A and B has $\mu_{[A,B]}(111212221) \neq 0$ and no non-vanishing Milnor invariants of lower weight as calculated in the proof of this proposition. Moreover, a simple computation of the Gassner representation as in Section 7 of \cite{klw} shows that neither $A$ nor $B$ lie in $\mathcal{P}(2)$, therefore the existence of this example proves the theorem for $n=2$. Note that this provides an alternate proof of the result of \cite{decampos}, avoiding theta graphs entirely. 

Now, consider the link $\sigma \subset \mathcal{C}(n)$ whose first two strands are isotopic to $C$ and whose remaining $n-2$ strands are trivial. This link is also a commutator, and we immediately see $\mu_{\widehat{\sigma}}(111212221) \neq 0$ therefore the result follows for all $n$.

\end{proof}

\knotconcsubgroup*
\begin{proof}
The result follows from a series of simple observations. First, note that any string link of the form $K_i$ which is split with $i^{th}$ strand isotopic to a knot $K \subset S^3$ cut open at a point and all other strands trivial will commute with any string link $\sigma$ by simply sliding $K$ along the $i^{th}$ strand of $[K_i, \sigma]$. 
Furthermore, note that $K_i \cdot J_l$ is $(K \# J)_i$ if $i=l$ and, if $i \neq l$, it is the split link with $i^{th}$ component isotopic to $K$ split open at a point and $l^{th}$ component isotopic to $J$ split open at a point. Therefore, the set $\{K_i \mid K \in \mathcal{C}, i \in \{1, ..., n\} \}$ generates $ \mathop{\oplus}\limits^n \ \mathcal{C}$. Furthermore, no element in this set is in $Ncl(\mathcal{P}(n))$ as each component of an element of $Ncl(\mathcal{P}(n))$ must have slice components (this is clear as pure braids have unknotted components).

\end{proof}

\subsection{Acknowledgments} The author was supported by an NSF Graduate Research Fellowship under Grant No. 1450681 as well as an AAUW American Fellowships Dissertation Fellowship. Additionally, the author was partially supported by NSF grant DMS-1745583 as a postdoc. The author is deeply grateful to her advisor, Shelly Harvey, for her guidance. The author would further like to express her appreciation for Tim Cochran for giving her a roadmap of where to look to understand Milnor's invariants and telling her to physically make links using household objects to check her more unwieldy linking number computations. The author would like to thank Jennifer Hom for her mentorship and generous feedback and would also like to thank her REU student Benjamin Pagano for pointing out a computational error in an earlier version of this paper. Finally, the author would like to thank Jean-Baptiste Meilhan and Akira Yasuhara for helpful correspondence.

\bibliographystyle{alpha} 
\bibliography{purebraids} 

\end{document}